\documentclass[a4paper,12pt,reqno]{amsart}
\usepackage{lmodern,amsmath,amssymb,mathtools,amsthm}
\usepackage{enumerate}
\usepackage{stackengine}
\usepackage{bigints}
\usepackage{geometry}
\usepackage{wasysym}
\usepackage{floatrow,morefloats}
\usepackage{graphicx}
\usepackage{caption}
\usepackage{subcaption}
\usepackage{floatrow,morefloats}
\usepackage{hyperref}
\usepackage{epstopdf}
\usepackage{longtable}
\usepackage{enumitem}  
\usepackage{caption,subcaption}
\newtheorem{lemma}{Lemma}
\newtheorem{theorem}{Theorem}

\newtheorem{proposition}{Proposition}
\numberwithin{equation}{section}
\numberwithin{definition}{section}
\numberwithin{lemma}{section}
\numberwithin{theorem}{section}
\numberwithin{corollary}{section}
\numberwithin{example}{section}
\numberwithin{proposition}{section}

\usepackage{blindtext}
\usepackage{mathrsfs}

\title{Geronimus transformation and Sobolev-type orthogonal polynomials}

\author{N. Neha$^\dagger$}
\address{$^\dagger$Department of Mathematics, Indian Institute of Technology, Roorkee-247667, Uttarakhand, India}
\email{neha@ma.iitr.ac.in, neharani0777@gmail.com}

\allowdisplaybreaks
\bigskip
\newmuskip\pFqmuskip

\catcode`,\active

\catcode`\,12

\begin{document}
	
	\sloppy
	
	\begin{abstract}
		Iterated Geronimus transformations generate Sobolev-type orthogonal polynomials from classical families. We establish a direct equivalence between a Sobolev inner product involving point evaluation and the first derivative at $a$ (located outside the support of the original measure) and two successive Geronimus transformations. Explicit three-term and five-term recurrence relations are derived for the double Geronimus polynomials, revealing their underlying algebraic structure. Connection formulas linking the Sobolev-orthogonal polynomials $Q_n^{M,N}(x)$ with both the original polynomials and the transformed Geronimus polynomials are obtained through Christoffel–Darboux kernels and determinantal representations. In the Jacobi case, a detailed asymptotic analysis shows that the ratios of derivatives and norms converge to explicit constants independent of the perturbation parameters $M$ and $N$. These results provide a unified framework connecting spectral transformation theory with Sobolev orthogonality.
	\end{abstract}
	
	
	\maketitle
	
	\markboth{N. Neha}{Double Geronimus transformation and Asymptotics}
	
\section{Introduction}
Orthogonal polynomials and their spectral transformations lie at the heart of approximation theory, integrable systems, and quantum mechanics. While classical orthogonal polynomial families (Jacobi, Laguerre, Hermite) are well-characterized through their weight functions and recurrence relations, modern applications in spectral theory \cite{Ismail book 2005} and Sobolev orthogonal polynomials \cite{Marcellan_Xu_2015} demand controlled perturbations of their orthogonality structures. Among such perturbations, the Geronimus transformation – first introduced in 1943 \cite{Geronimus 1943} but overlooked until its rediscovery in integrable lattice systems \cite{Spiridonov Zhedanov Methods Appl. Anal. 1995} – has emerged as a fundamental tool for constructing new orthogonal polynomial sequences through measure modification.

\vspace{0.3em}
Iterated Geronimus transformations applied to classical orthogonal polynomials generate Sobolev-type orthogonal families by systematically incorporating derivative terms into the inner product through layered perturbations. A single Geronimus transformation at a point $a \notin \text{supp}(\mu)=:E$ modifies the measure as $d\mu_g(x) = \frac{1}{x-a} d\mu(x) + M \delta(x-a$), altering orthogonality conditions but not involving derivatives. 
To account for derivatives, the transformation must be iterated at the same point $a$. This repeated process builds rational factors of the form $\frac{1}{(x-a)^k}$ and adds multiple Dirac masses, ultimately leading to an inner product that includes discrete evaluations of the functions and their higher-order derivatives at $a$.

\vspace{0.3em}
A double Geronimus transformation \cite{Derevyagin Marcellan Numer. Algorithms 2014} addresses this limitation by inducing a Sobolev inner product of the form
\begin{align}\label{Sobolev inner product-Geronimus}
	\langle p,q\rangle_{gg} = \int_{E}p(x)q(x)d\mu_{gg}(x) + \mathbf{p}(a)^\top \mathbf{M}_1 \mathbf{q}(a),
\end{align}
where $\textbf{p}(a)=\begin{bmatrix}
	p(a) & p'(a)
\end{bmatrix}^{T}$ and $\mathbf{M}_1$ encodes masses and moments as
\begin{align*}
	\begin{bmatrix}
		b & c\\
		c & d
	\end{bmatrix}-\begin{bmatrix}
		\int_{E}d\mu_{gg}(x) & \int_{E}xd\mu_{gg}(x)\\
		\int_{E}xd\mu_{gg}(x) & \int_{E}x^2d\mu_{gg}(x)
	\end{bmatrix}.
\end{align*}
The subtraction of the moment terms in $\mathbf{M}_1$ guarantees that orthogonality is preserved, even though the inner product combines an integral part with discrete evaluations at the point $a$. Crucially, the transformed measure satisfies $(x-a)^2d\mu_{gg}=d\mu$, maintaining algebraic consistency with the original measure $\mu$. The resulting orthogonal polynomials $\{P_n^{gg}(x)\}_{n\geq0}$ exhibit the relation in terms of $\{P_n(x)\}_{n\geq0}$, which are orthogonal with respect to $\mu$,
\begin{align*}
	P_n^{gg}(x)= P_n(x)+B_nP_{n-1}(x)+C_nP_{n-2}(x),\; n\geq 1,\; C_n\neq0\;\text{for}\;n\geq2.
\end{align*}
This structure reflects the non-diagonal Sobolev inner product, inducing higher-order recurrence relations distinct from classical three-term recursions \cite{Alfaro Pena Rezola Marcellan Comput. Math. Appl. 2011,Hounga Hounkonnou Ronveaux J. Comput. Appl. Math. 2006}. Furthermore, from \eqref{Sobolev inner product-Geronimus}, we have $\langle (x-a)^2p,q\rangle_{gg}=\langle p,q\rangle_{0}:=\int_{E}p(x)q(x)d\mu(x)$. These types of families are well explored in the literature (see \cite{Alfaro Pena Rezola Marcellan Comput. Math. Appl. 2011, Hounga Hounkonnou Ronveaux J. Comput. Appl. Math. 2006} and reference therein). For further iterations of the Geronimus transformation, see the results presented in \cite{Alfaro Pena Rezola Marcellan J. Comput. Appl. Math. 2010, Kwon Lee Marcellan Park Ann. Mat. Pura Appl. 2001}. The Geronimus introduced the transformation in \cite{Geronimus 1943, Geronimus 1961} as a method for constructing new sequences of orthogonal polynomials from existing families of known sequences. It has also been examined in connection with different schemes of mechanical quadratures (see \cite{Shohat Trans. Amer. Math. Soc. 1937}). Further references on the Geronimus transformation can be found in \cite{Bueno Marcellan Linear Algebra Appl. 2004} and \cite{Bueno Deano Tavernetti Numer. Algorithms 2010}. The case of multiple Geronimus transformations is studied in \cite{Derevyagin Garcia Marcellan Linear Algebra Appl. 2014}, where it is shown that such iterations lead to discrete Sobolev-type inner products.
 
\vspace{0.3em}
Sobolev orthogonality often arises from spectral transformations (e.g., Darboux or Geronimus) that perturb classical measures. For instance, the Krall-Jacobi polynomials, arising from iterated Darboux transformations of classical Jacobi polynomials \cite{Iliev J. Math. Pures Appl. 2011}, exemplify how spectral methods generate Sobolev orthogonal systems. These polynomials satisfy higher-order differential equations, reflecting the action of commutative algebras of operators, a structure mirrored in Sobolev inner products that incorporate derivatives as perturbations. The iterated Christoffel transformations \cite{Hermoso Huertas Lastra Marcellan Numer. Algorithms 2023} modify moments in ways equivalent to introducing point masses or derivative couplings, as seen in coherent pairs \cite{Marcellan_Xu_2015}. These systems reflect a unifying structure where spectral perturbations (point masses, derivatives) generate commutative algebras of differential operators, as seen in bispectral frameworks \cite{Grunbaum Haine Horozov 1999}. This synthesis underscores a spectral-theoretic duality: perturbations via spectral transformations (Geronimus, Darboux, Uvarov) inherently encode derivative terms into the orthogonality structure, as demonstrated by Krall-type systems \cite{Alvarez_Arvesu_Marcellan_2004}. The Sobolev orthogonality characteristic of these systems reveals a unifying spectral framework, wherein perturbations manifest through point masses or modifications of moments, thereby introducing derivative components \cite{Alfaro Pena Rezola Marcellan Comput. Math. Appl. 2011, Hounga Hounkonnou Ronveaux J. Comput. Appl. Math. 2006}.

\vspace{0.3em}
In this manuscript, we examine the relationship between iterated Geronimus transformations and Sobolev-type orthogonal polynomials generated by inner products that include point evaluations and derivatives. In particular, we consider the monic polynomial sequence $\{Q_n^{M,N}(x)\}_{n=0}^{\infty}$ that is orthogonal with respect to the Sobolev inner product
\begin{align}\label{Sobolev inner product}
	\langle p,q\rangle_{S}=\int_{E}p(x)q(x)d\mu(x)+Mp(a)q(a)+Np'(a)q'(a),\quad a\notin \text{supp}(\mu),
\end{align}
defined on the space $\mathbb{P}$ of polynomials with real coefficients. This inner product introduces a double perturbation at $x=a$, combining both function values and first derivatives—a structure that naturally arises in problems involving spectral data modifications or boundary conditions. 
 
\vspace{0.3em}
Building on foundational work from \cite{Bavinck_1998}, we observe that the polynomials $Q_n^{M,N}(x)$ can be expressed in terms of the original orthogonal polynomials $P_n(x)$ associated with $\mu$, analogous to the explicit relations derived for classical orthogonal polynomials under iterated Geronimus transformations (see \cite{Derevyagin Marcellan Numer. Algorithms 2014}). We show that applying two successive Geronimus transformations to the polynomial sequence $\{P_n(x)\}_{n\geq0}$ produces the Sobolev orthogonal polynomials $Q_n^{M,N}(x)$, thereby establishing a direct connection between discrete spectral transformations and Sobolev inner products involving derivatives. The primary aim of this work is to analyze the interplay between the Sobolev-modified polynomials $Q_n^{M,N}(x)$ and their iteratively transformed counterparts $P_n^{gg}(x)$ at the critical point $x=a$. Unlike previous studies, which primarily focus on single-step transformations, our work explores the recurrence relations and asymptotic properties of orthogonal systems generated through iterated Geronimus transformations.

\vspace{0.3em}
The explicit formula and matrix representation of the double Geronimus polynomials were presented in \cite{Derevyagin Marcellan Numer. Algorithms 2014}. This manuscript extends the study by exploring additional properties of the double Geronimus polynomials and examining how successive Geronimus transformations influence the structure and asymptotic behavior of Sobolev-type orthogonal polynomials.
The remainder of this manuscript is as follows: In Section~\ref{Relations for double Geronimus transformation polynomials}, we derive recurrence relations for the double Geronimus polynomials $\{P_n^{gg}(x)\}$, including both three-term and higher-order forms induced by the quadratic factor $(x-a)^2$, revealing their algebraic and orthogonality structure. Section~\ref{connection formulas} introduces a Sobolev inner product perturbed by point evaluations and derivatives, establishes explicit connection formulas between the Sobolev-orthogonal polynomials $Q_n^{M, N}(x)$ and $P_n^{gg}(x)$ through kernel polynomials and determinant representations. These results clarify how discrete spectral modifications at $x = a$ influence the Sobolev framework. Section~\ref{asymptotics} presents the asymptotic behavior of the Jacobi polynomials $P_n^{(\alpha,\beta)}(x)$ and the associated Sobolev-type orthogonal polynomials $Q_n^{M,N}(x)$. It includes detailed comparisons of norms and the study of their behavior under successive transformations.

\section{Relations for double Geronimus transformation polynomials}\label{Relations for double Geronimus transformation polynomials}
As we know, orthogonal polynomials \cite{Chihara Book 1978} satisfy a three-term recurrence relation given by
\begin{align}\label{Three-term recurrence relation}
		xP_n(x)=P_{n+1}(x)+\beta_n P_{n}(x)+\gamma_n P_{n-1}(x), \quad n\geq0.
\end{align}
Since the polynomials $\{P_n^{gg}(x)\}$ are also orthogonal, it is natural to explore the recurrence relation they satisfy and how their recurrence coefficients depend on $\beta_n$ and $\gamma_n$. We also derive alternative recurrence relations satisfied by $\{P_n^{gg}(x)\}$.

\begin{theorem}
The double Geronimus polynomials $\{P_n^{gg}(x)\}_{n\geq0}$ satisfy the three-term recurrence relation
	\begin{align}\label{three-term recurrence relation for Double Geronimus polynomials}
		P_{n+1}^{gg}(x)=(x-a-\sigma_{n,n}^{gg})P_n^{gg}(x)-\sigma_{n,n-1}^{gg}P_{n-1}^{gg}(x),\quad n\geq0,
	\end{align}
	where
	\begin{align*}
		\sigma_{n,n}^{gg}=B_n-B_{n+1}+(\beta_n-a)\quad{\mbox{and}}\quad \sigma_{n,n-1}^{gg}=\frac{C_n\lVert P_{n-2}\rVert^2_{0}}{C_{n-1}\lVert P_{n-3}\rVert^2_{0}}.
	\end{align*}
	\begin{proof}
	Since $\{P_n^{gg}(x)\}_{n \geq 0}$ forms a standard orthogonal sequence, we can write
		\begin{align*}
			(x-a)P_n^{gg}(x)=P_{n+1}^{gg}(x)+\sigma_{n,n}^{gg}P_n^{gg}(x)+\sigma_{n,n-1}^{gg}P_{n-1}^{gg}(x),
		\end{align*}
with coefficients given by
	\begin{align*}
		\sigma_{n,n}^{gg}=\frac{\langle(x-a)P_n^{gg}(x),P_n^{gg}(x)\rangle_{gg}}{\langle P_n^{gg}(x),P_n^{gg}(x)\rangle_{gg}} \quad{\mbox{and}}\quad\sigma_{n,n-1}^{gg}=\frac{\langle(x-a)P_n^{gg}(x),P_{n-1}^{gg}(x)\rangle_{gg}}{\langle P_{n-1}^{gg}(x),P_{n-1}^{gg}(x)\rangle_{gg}}.
	\end{align*}
From the expansion
\[
P_n^{gg}(x) = P_n(x) + B_n P_{n-1}(x) + C_n P_{n-2}(x),
\]
we compute
\begin{align*}
	\langle (x-a) P_n^{gg}, P_n^{gg} \rangle_{gg}
	&= \int_E P_n^{gg}(x) (x - a) P_n(x) \, d\mu_{gg}(x) \\
	&\quad + B_n \int_E P_n^{gg}(x)(x - a) P_{n-1}(x) \, d\mu_{gg}(x) \\
	&\quad + C_n \int_E P_n^{gg}(x)(x - a) P_{n-2}(x) \, d\mu_{gg}(x).
\end{align*}
Using the recurrence relation \eqref{Three-term recurrence relation}
\[
(x - a) P_n(x) = P_{n+1}(x) + (\beta_n - a) P_n(x) + \gamma_n P_{n-1}(x),
\]
we get
\begin{align*}
	\langle (x - a) P_n^{gg}, P_n^{gg} \rangle_{gg}
	&= \int_E P_n^{gg}(x) P_{n+1}(x) \, d\mu_{gg}(x) + (\beta_n - a) \int_E P_n^{gg}(x) P_n(x) \, d\mu_{gg}(x) \\
	&\quad + B_n \int_E P_n^{gg}(x) P_n(x) \, d\mu_{gg}(x) \\
	&= \int_E P_n^{gg}(x) \left( P_{n+1}^{gg}(x) - B_{n+1} P_n(x) - C_{n+1} P_{n-1}(x) \right) d\mu_{gg}(x) \\
	&\quad + (B_n + \beta_n - a) \|P_n^{gg}\|^2_{gg} \\
	&= (B_n - B_{n+1} + \beta_n - a) \|P_n^{gg}\|^2_{gg}.
\end{align*}
Similarly,
\begin{align*}
	\langle (x - a) P_n^{gg}(x), P_{n-1}^{gg}(x) \rangle_{gg}
	= \int_E P_n^{gg}(x) (x - a) P_{n-1}^{gg}(x) \, d\mu_{gg}(x)
	= \|P_n^{gg}\|^2_{gg}.
\end{align*}
Moreover, the norm satisfies
\begin{align}\label{eq:norm_relation}
	\|P_k^{gg}\|^2_{gg} = \langle (x - a)^2 P_{k-2}^{gg}(x), P_k^{gg}(x) \rangle_{gg} = \langle P_{k-2}^{gg}(x), P_k^{gg}(x) \rangle_0 = C_k \|P_{k-2}\|^2_0.
\end{align}
Hence, we achieve the required result.
	\end{proof}
\end{theorem}
For the Jacobi measure with $a = -1$, the double Geronimus-modified measure becomes
\[
d\mu_{gg}(x) = (1 - x)^{\alpha} (1 + x)^{\beta - 2}.
\]
Let $\{P_n^{(\alpha,\beta)}(x)\}_{n=0}^\infty$ denote the monic Jacobi polynomials orthogonal with respect to $d\mu(x)\\= (1 - x)^{\alpha} (1 + x)^{\beta}$, and let $\{P_n^{(\alpha,\beta - 2)}(x)\}_{n=0}^\infty$ be orthogonal with respect to $d\mu_{gg}(x)$.
Then the relationship between these polynomials is
\begin{align*}
	P_n^{(\alpha,\beta - 2)}(x)
	&= P_n^{(\alpha,\beta)}(x) + \frac{4n(\alpha + n)}{(2n + \alpha + \beta - 2)(2n + \alpha + \beta)} P_{n-1}^{(\alpha,\beta)}(x) \\
	&\quad + \frac{4n(n-1)(\alpha + n)(\alpha + n - 1)}{(2n + \alpha + \beta - 3)(2n + \alpha + \beta - 2)^2(2n + \alpha + \beta - 1)} P_{n-2}^{(\alpha,\beta)}(x).
\end{align*}
For the monic Jacobi polynomials $P_n^{(\alpha,\beta)}(x)$ with $a = -1$, the recurrence relation \eqref{three-term recurrence relation for Double Geronimus polynomials} simplifies to
\begin{align*}
	&xP_{n}^{(\alpha,\beta-2)}(x)=P_{n+1}^{(\alpha,\beta-2)}(x)+\frac{(\beta-2)^2-\alpha^2}{(2n+\alpha+\beta-2)(2n+\alpha+\beta)}P_{n}^{(\alpha,\beta-2)}(x)\\
	&+\frac{4n(\alpha+n)(n+\beta-2)(n+\alpha+\beta-2)}{(2n+\alpha+\beta-3)(2n+\alpha+\beta-2)^2(2n+\alpha+\beta-1)}P_{n-1}^{(\alpha,\beta-2)}(x),\quad\beta>1,\;\alpha>-1.
\end{align*}
\vspace{0.5em}
\begin{proposition}
	The double Geronimus polynomials $\{P_n^{gg}(x)\}_{n=0}^{\infty}$ satisfy the five-term recurrence relation
	\begin{align*}
		(x-a)^2P_n^{gg}(x)&=P_{n+2}^{gg}(x)+\alpha_{n+1,n+1}^{gg}P_{n+1}^{gg}(x)+\alpha_{n+1,n}^{gg}P_n^{gg}(x)+\alpha_{n+1,n-1}^{gg}P_{n-1}^{gg}(x)\\
		&+\alpha_{n+1,n-2}^{gg}P_{n-2}^{gg}(x),
	\end{align*}
where the coefficients are given by
\begin{align*}
	&\alpha_{n+1,n+1}^{gg}=\frac{C_{n+1}B_n\lVert P_{n-1}\rVert^2_{0}+B_{n+1}\lVert P_n\rVert^2_{0}}{C_{n+1}\lVert P_{n-1}\rVert^2_{0}},\quad\alpha_{n+1,n}^{gg}=\frac{\lVert P_n\rVert^2_{0}+B_n^2\lVert P_{n-1}\rVert^2_{0}+C^2_n\lVert P_{n-2}\rVert^2_{0}}{C_n\lVert P_{n-2}\rVert^2_{0}},\\
	&\alpha_{n+1,n-1}^{gg}=\frac{B_n\lVert P_{n-1}\rVert^2_{0}+B_{n-1}C_n\lVert P_{n-2}\rVert^2_{0}}{C_{n-1}\lVert P_{n-3}\rVert^2_{0}},\quad\alpha_{n+1,n-2}^{gg}=\frac{C_n\lVert P_{n-2}\rVert^2_{0}}{C_{n-2}\lVert P_{n-4}\rVert^2_{0}}.
\end{align*}
	\begin{proof}
	Since $(x-a)^2P_n^{gg}(x)$ is a polynomial of degree $n+2$ and the family $\{P_k^{gg}(x)\}_{k\geq 0}$ forms a basis, we can express
		\begin{align*}
			(x-a)^2P_n^{gg}(x)=P_{n+2}^{gg}(x)+\sum_{k=0}^{n+1}\alpha_{n+1,k}^{gg}P_k^{gg}(x),
		\end{align*}
	where
	\begin{align*}
		\alpha_{n+1,k}^{gg}=\frac{\langle(x-a)^2P_n^{gg}(x),P_k^{gg}(x)\rangle_{gg}}{\langle P_k^{gg}(x),P_k^{gg}(x)\rangle_{gg}}.
	\end{align*}
By orthogonality, $\langle (x-a)^2P_n^{gg}(x), P_k^{gg}(x) \rangle_{gg} = 0$ for $k \leq n - 3$, reducing the expansion to
\begin{align*}
	(x-a)^2P_n^{gg}(x)&=P_{n+2}^{gg}(x)+\alpha_{n+1,n+1}^{gg}P_{n+1}^{gg}(x)+\alpha_{n+1,n}^{gg}P_n^{gg}(x)+\alpha_{n+1,n-1}^{gg}P_{n-1}^{gg}(x)\\
	&+\alpha_{n+1,n-2}^{gg}P_{n-2}^{gg}(x).
\end{align*}
We now compute the numerators in the coefficients using the structure of $P_n^{gg}(x)$ in terms of the original orthogonal polynomials $P_n(x)$
\begin{align*}
	\langle (x-a)^2 P_n^{gg}, P_{n+1}^{gg} \rangle_{gg} 
	&= \int_E (x-a)^2 P_n^{gg}(x) P_{n+1}^{gg}(x) \, d\mu_{gg}(x)= \int_E P_n^{gg}(x) P_{n+1}^{gg}(x) \, d\mu(x) \\
	&= \int_E P_n^{gg}(x) \left( P_{n+1}(x) + B_{n+1}P_n(x) + C_{n+1}P_{n-1}(x) \right) d\mu(x) \\
	&= B_{n+1} \lVert P_n \rVert_0^2 + C_{n+1} B_n \lVert P_{n-1} \rVert_0^2.
\end{align*}
Similarly, we compute the remaining inner products:
\begin{align*}
	\langle (x-a)^2 P_n^{gg}(x), P_n^{gg}(x) \rangle_{gg} &= \lVert P_n \rVert_0^2 + B_n^2 \lVert P_{n-1} \rVert_0^2 + C_n^2 \lVert P_{n-2} \rVert_0^2, \\
	\langle (x-a)^2 P_n^{gg}(x), P_{n-1}^{gg}(x) \rangle_{gg} &= B_n \lVert P_{n-1} \rVert_0^2 + B_{n-1}C_n \lVert P_{n-2} \rVert_0^2, \\
	\langle (x-a)^2 P_n^{gg}(x), P_{n-2}^{gg}(x) \rangle_{gg} &= C_n \lVert P_{n-2} \rVert_0^2.
\end{align*}
Substituting these expressions into the definition of $\alpha_{n+1,k}^{gg}$ yields the desired recurrence coefficients.
	\end{proof}
\end{proposition}
\section{Connection formulas}\label{connection formulas}
Building on the recurrence structure established for the double Geronimus polynomials $P_n^{gg}(x)$ in Section~\ref{Relations for double Geronimus transformation polynomials}, we now employ the connection formulas derived in \cite{Hermoso Huertas Lastra Marcellan Numer. Algorithms 2023} to express the Sobolev-type orthogonal polynomials $Q_n^{M,N}(x)$ in terms of the classical Jacobi polynomials $ P_n^{(\alpha,\beta)}(x) $. In particular, by leveraging the Christoffel-Darboux kernel and the determinantal representations introduced in \cite{Hermoso Huertas Lastra Marcellan Numer. Algorithms 2023}, we establish a bridge between the recursive framework of $P_n^{gg}(x)$ and the perturbed orthogonality structure of $Q_n^{M,N}(x)$, adapting these analytical tools to the iterated Geronimus-Sobolev context.

\vspace{0.5em}
The Sobolev-type orthogonal polynomials $Q_n^{M,N}(x)$ admit the following representation in terms of the orthogonal polynomials $P_n(x)$ (see \cite[Lemma 1]{Hermoso Huertas Lastra Marcellan Numer. Algorithms 2023})
\begin{align}\label{Sobolev orthogonal}
	Q_n^{M,N}(x)=P_n(x)-MQ_n^{M,N}(a)K_{n-1}(x,a)-N(Q_n^{M,N})'(a)K_{n-1}^{(0,1)}(x,a),
\end{align}
where $K_n(x,y)$ denotes the Christoffel--Darboux kernel, given by
\begin{align*}
	K_n(x,y)=\sum_{k=0}^{n}\frac{P_k(x)P_k(y)}{\lVert P_k(x)\rVert^2_{0}}.
\end{align*}
To compute the unknown values \( Q_n^{M,N}(a) \) and \( (Q_n^{M,N})'(a) \), we proceed by evaluating equation~\eqref{Sobolev orthogonal} at \( x = a \), and then differentiating it with respect to \( x \), followed by evaluation at \( x = a \),
\begin{align*}
	Q_n^{M,N}(a)&=P_n(a)-MQ_n^{M,N}(a)K_{n-1}(a,a)-N(Q_n^{M,N})'(a)K_{n-1}^{(0,1)}(a,a),\\
	(Q_n^{M,N})'(a)&=P_n'(a)-MQ_n^{M,N}(a)K_{n-1}^{(1,0)}(a,a)-N(Q_n^{M,N})'(-1)K_{n-1}^{(1,1)}(a,a).
\end{align*}
Solving this system of equations yields closed-form expressions for the evaluation and derivative of \( Q_n^{M,N}(x) \) at \( x = a \),
\begin{align}
	Q_n^{M,N}(a)&=\frac{\begin{vmatrix}
			P_n(a)&NK_{n-1}^{(0,1)}(a,a)\\
			P_n'(a)&1+NK_{n-1}^{(1,1)}(a,a)
	\end{vmatrix}}{\begin{vmatrix}
			1+MK_{n-1}(a,a)&NK_{n-1}^{(0,1)}(a,a)\\
			MK_{n-1}^{(1,0)}(a,a)&1+NK_{n-1}^{(1,1)}(a,a)
	\end{vmatrix}},\label{value of sobolev polynomials at -1}\\
	(Q_n^{M,N})'(a)&=\frac{\begin{vmatrix}
			1+MK_{n-1}(a,a)&P_n(a)\\
			MK_{n-1}^{(1,0)}(a,a)	&P_n'(a)
	\end{vmatrix}}{\begin{vmatrix}
			1+MK_{n-1}(a,a)&NK_{n-1}^{(0,1)}(a,a)\\
			MK_{n-1}^{(1,0)}(a,a)&1+NK_{n-1}^{(1,1)}(a,a)
	\end{vmatrix}}.\label{value of derivative of sobolev polynomials at -1}
\end{align}
Substituting these expressions into \eqref{Sobolev orthogonal} yields the determinantal formula for \( Q_n^{M,N}(x) \) in terms of the classical polynomials
\begin{align*}
	Q_n^{M,N}(x)=\frac{\begin{vmatrix}
			P_n(x) & M K_{n-1}(x,a) & N K_{n-1}^{(0,1)}(x,a)\\
			P_n(a) & 1+M K_{n-1}(a,a) & N K_{n-1}^{(0,1)}(a,a)\\
			P_n'(a) & M K_{n-1}^{(0,1)}(a,a) & 1+N K_{n-1}^{(1,1)}(a,a)
			\end{vmatrix}}{\begin{vmatrix}
			1+M K_{n-1}(a,a) & N K_{n-1}^{(0,1)}(a,a)\\
			MK_{n-1}^{(0,1)}(a,a) & 1+N K_{n-1}^{(1,1)}(a,a)
		\end{vmatrix}}.
\end{align*}

The following theorem establishes a critical connection between the Sobolev-type orthogonal polynomials \( \{Q_n^{M,N}(x)\} \) and the iteratively transformed Geronimus-Jacobi polynomials \( \{P_n^{gg}(x)\} \).
\begin{theorem}
	The Sobolev-type orthogonal polynomials $\{Q_n^{M,N}(x)\}_{n\geq0}$ and the double Geronimus polynomials $\{P_n^{gg}(x)\}_{n\geq0}$ satisfy the recurrence relation
	\begin{align*}
		(x-a)^2Q_n^{M,N}(x)&=P_{n+2}^{gg}(x)+\alpha_{n+1,n+1}P_{n+1}^{gg}(x)+\alpha_{n+1,n}P_{n}^{gg}(x)+\alpha_{n+1,n-1}P_{n-1}^{gg}(x)\\
		&+\alpha_{n+1,n-2}P_{n-2}^{gg}(x),
	\end{align*}
where the coefficients \( \alpha_{n+1,k} \) are defined as follows
\begin{align*}
	\alpha_{n+1,n+1}&=\frac{B_{n+1}\lVert P_n\rVert^2_{0}-C_{n+1}(MQ_n^{M,N}(a)P_{n-1}(a)
		+N(Q_n^{M,N})'(a)P_{n-1}'(a))}{C_{n+1}\lVert P_{n-1}\rVert^2_{0}},\\
	\alpha_{n+1,n}&=\frac{1}{C_n\lVert P_{n-2}\rVert^2_{0}}\big(\lVert P_n\rVert^2_{0}-\big(MQ_n^{M,N}(a)P_{n-1}(a)+N(Q_n^{M,N})'(a)P'_{n-1}
	(a)\big)B_n\\
	&-\big(MQ_n^{M,N}(a)P_{n-2}(a)+N(Q_n^{M,N})'(a)(P_{n-2})'(a)\big)C_n\big),\\
	\alpha_{n+1,n-1}&=\frac{1}{C_{n-1}\lVert P_{n-3}\rVert^2_{0}}\big(-MQ_n^{M,N}(a)(P_{n-1}
	(a)+B_{n-1}P_{n-2}(a)-C_{n-1}P_{n-3}(a))\\
	&-N(Q_n^{M,N})'(a)(P'_{n-1}(a)+B_{n-1}P_{n-2}'(a)-C_{n-1}P_{n-3}'(a))\big),\\
	\alpha_{n+1,n-2}&=\frac{1}{C_{n-2}\lVert P_{n-4}\rVert^2_{0}}\big(-MQ_n^{M,N}(a)(P_{n-2}
	(a)+B_{n-2}P_{n-3}(a)-C_{n-2}P_{n-4}(a))\\
	&-N(Q_n^{M,N})'(a)(P'_{n-2}(a)+B_{n-2}(P_{n-3})'(a)-C_{n-2}(P_{n-4})'(a))\big).
\end{align*}
\begin{proof}
We begin by expressing the expansion of \((x-a)^2 Q_n^{M,N}(x)\) in terms of the orthogonal basis \( \{P_k^{gg}(x)\} \)
	\begin{align*}
			(x-a)^2Q_n^{M,N}(x)=P_{n+2}^{gg}(x)+\sum_{k=0}^{n+1}\alpha_{n+1,k}P_k^{gg}(x),
	\end{align*}
where
\begin{align*}
	\alpha_{n+1,k}=\frac{\langle (x-a)^2Q_n^{M,N}(x),P_k^{gg}(x)\rangle_{gg}}{\langle P_k^{gg}(x),P_k^{gg}(x)\rangle_{gg}}.
\end{align*}
We now use the Sobolev inner product \eqref{Sobolev inner product} associated with the measure \(\mu_{gg}\)
\begin{align*}
	\langle p,q\rangle_{S}=\int_{E}p(x)q(x)(x-a)^2d\mu_{gg}(x)+Mp(a)q(a)+Np'(a)q'(a).
\end{align*}
Substituting \( p(x) = Q_n^{M,N}(x) \) and \( q(x) = (x-a)^2 P_k^{gg}(x) \), we get
\begin{align*}
     \langle (x-a)^2Q_n^{M,N}(x),P_k^{gg}(x)\rangle_{gg}&\!=\!
     \langle Q_n^{M,N}(x),(x-a)^2P_k^{gg}(x)\rangle_{gg}\!=\!\langle Q_n^{M,N}(x),(x-a)^2P_k^{gg}(x)\rangle_{S}\\
     &=0 \quad{\mbox{for}}\quad k\leq n-3.
\end{align*}
Thus, we achieve
\begin{align*}
	(x-a)^2Q_n^{M,N}(x)&=P_{n+2}^{gg}(x)+\alpha_{n+1,n+1}P_{n+1}^{gg}(x)+\alpha_{n+1,n}P_{n}^{gg}(x)+\alpha_{n+1,n-1}P_{n-1}^{gg}(x)\\
	&+\alpha_{n+1,n-2}P_{n-2}^{gg}(x).
\end{align*}
To compute the coefficients, we proceed term by term. For \(\alpha_{n+1,n+1}\), we use the structure of \(P_{n+1}^{gg}(x)\)
\[
P_{n+1}^{gg}(x) = P_{n+1}(x) + B_{n+1} P_n(x) + C_{n+1} P_{n-1}(x),
\]
and since \(\langle Q_n^{M,N}(x), P_{n+1}(x) \rangle_0 = 0\), we find
\begin{align*}
	\langle (x-a)^2Q_n^{M,N}(x),P_{n+1}^{gg}(x)\rangle_{gg}
	&=B_{n+1}\lVert P_n\rVert^2_{0}+C_{n+1}\langle Q_n^{M,N}(x),P_{n-1}(x)\rangle_{0}.
\end{align*}
Using the known structure of \(Q_n^{M,N}(x)\) and kernel identities, we compute
\[
\langle Q_n^{M,N}(x), P_{n-1}(x) \rangle_0 = - M Q_n^{M,N}(a) P_{n-1}(a) - N (Q_n^{M,N})'(a) P_{n-1}'(a),
\]
from which \(\alpha_{n+1,n+1}\) follows directly.

The computations for \(\alpha_{n+1,n}, \alpha_{n+1,n-1}, \alpha_{n+1,n-2}\) proceed analogously using the recurrence relations and inner product identities. Each involves combining the orthogonality of \(P_n(x)\) with the structure of \(Q_n^{M,N}(x)\), and evaluating via the values and derivatives at \(x = a\).
\end{proof}
\end{theorem}
\begin{proposition}
Let $\{P_n^{gg}(x)\}_{n \geq 0}$ be the sequence of orthogonal polynomials associated with the double Geronimus transformation of the measure, and let $\{P_n(x)\}_{n \geq 0}$ be the standard orthogonal polynomial sequence. Then, the following connection formula holds
	\begin{align}\label{relation between double Geronimus polynomials and standard orthogonal polynomials}
		(x-a)^2P_n(x)=P_{n+2}^{gg}(x)+\sigma_{n+1,n+1}P_{n+1}^{gg}(x)+\sigma_{n+1,n}P_{n}^{gg}(x), \quad n\geq 0,
	\end{align}
	where the connection coefficients are given by
	\begin{align*}
		\sigma_{n+1,n+1}=\frac{B_{n+1}\lVert P_{n}\rVert^2_{0}}{C_{n+1}\lVert P_{n-1}\rVert^2_{0}} \quad{\mbox{and}}\quad \sigma_{n+1,n}=\frac{\lVert P_{n}\rVert^2_{0}}{C_{n}\lVert P_{n-2}\rVert^2_{0}}.
	\end{align*}
	\begin{proof}
		Since $(x-a)^2 P_n(x)$ is a polynomial of degree $n+2$ and $\{P_n^{gg}(x)\}_{n \geq 0}$ forms a complete basis, we may expand
		\begin{align*}
			(x-a)^2P_n(x)=P_{n+2}^{gg}(x)+\sum_{k=0}^{n+1}\sigma_{n+1,k}P_{k}^{gg}(x),
		\end{align*}
		where the coefficients are given by the orthogonality relation
		\begin{align*}
			\sigma_{n+1,k}=\frac{\langle (x-a)^2P_n(x),P_k^{gg}(x)\rangle_{gg}}{\langle P_k^{gg}(x),P_k^{gg}(x)\rangle_{gg}}.
		\end{align*}
	Using the structure relation for $P_k^{gg}(x)$
	\[
	P_k^{gg}(x) = P_k(x) + B_k P_{k-1}(x) + C_k P_{k-2}(x),
	\]
	and orthogonality of $\{P_n(x)\}$, we have
		\begin{align*}
			\langle (x-a)^2P_n(x),P_k^{gg}(x)\rangle_{gg}&=\langle P_n(x),P_k^{gg}(x)\rangle_{0}=\langle P_n(x),P_k(x)+B_kP_{k-1}(x)+C_kP_{k-2}(x)\rangle_{0}\\
			&=0 \quad{\mbox{for}} \quad k\leq n-1,
		\end{align*}
		it follows that
		\begin{align*}
			(x-a)^2P_n(x)=P_{n+2}^{gg}(x)+\sigma_{n+1,n+1}P_{n+1}^{gg}(x)+\sigma_{n+1,n}P_{n}^{gg}(x).
		\end{align*}
	Computing the non-zero numerators explicitly
		\begin{align*}
			\langle (x-a)^2P_n(x),P_{n+1}^{gg}(x)\rangle_{gg}&\!=\!\langle P_n(x),P_{n+1}(x)+B_{n+1}P_{n}(x)+C_{n+1}P_{n-1}(x)\rangle_{0}\!=\!B_{n+1}\lVert P_{n}\rVert^2_{0},\\
			\langle (x-a)^2P_n(x),P_{n}^{gg}(x)\rangle_{gg}&=\langle P_n(x),P_{n}(x)+B_{n}P_{n-1}(x)+C_{n}P_{n-2}(x)\rangle_{0}=\lVert P_{n}\rVert^2_{0}.
		\end{align*}
	The result follows by substituting into the definition of $\sigma_{n+1,k}$.
	\end{proof}
\end{proposition}

\vspace{1em}
\noindent
For the special case of the monic Jacobi polynomials $P_n^{(\alpha,\beta)}(x)$ and the double Geronimus transformation applied at $a = -1$, identity \eqref{relation between double Geronimus polynomials and standard orthogonal polynomials} becomes
\begin{align*}
	(1+x)^2P_n^{(\alpha,\beta)}(x)&=P_{n+2}^{(\alpha,\beta-2)}(x)+\frac{4(n+\beta)(n+\alpha+\beta)}{(2n+\alpha+\beta)(2n+\alpha+\beta+2)}P_{n+1}^{(\alpha,\beta-2)}(x)\\
	&+\frac{4(n+\beta-1)(n+\beta)(n+\alpha+\beta-1)(n+\alpha+\beta)}{(2n+\alpha+\beta-1)(2n+\alpha+\beta)^2(2n+\alpha+\beta+1)}P_{n}^{(\alpha,\beta-2)}(x).
\end{align*}
\noindent
This identity reveals how the multiplication by $(1+x)^2$ induces a shift in the Jacobi weight exponent $\beta \mapsto \beta - 2$, expressed explicitly in terms of the transformed orthogonal basis.

\section{Asymptotics}\label{asymptotics}
This section establishes key asymptotic results for Sobolev-Jacobi polynomials 
$Q_n^{M,N}(x)$ and $P_n^{gg}(x)$, derived from iterated Geronimus transformations. These results are critical for understanding the stability and convergence of spectral methods in Sobolev Orthogonal polynomials.
\begin{lemma}\rm(Gamma function Scaling).
	For $k,l\in\mathbb{N}\cup\{0\}$,
	\begin{align}\label{Gamma equality}
		\lim_{n\to+\infty}\frac{n^{k-l}\Gamma(n+l)}{\Gamma(n+k)}=1.
	\end{align}
\end{lemma}

\begin{lemma}\label{Derivative and norm of Jacobi polynomials}\rm(Derivative and norm of Jacobi polynomials).
	For $k\in\mathbb{N}\cup\{0\}$, the scaled monic Jacobi polynomials $\tilde{P}_n^{(\alpha,\beta)}(x):=\frac{(n+\alpha+\beta+1)_n}{2^n(\alpha+1)_n}P_n^{(\alpha,\beta)}(x)$ satisfy
	\begin{align}\label{Asymptotes of Jacobi at -1}
		\lim_{n\to+\infty}(-1)^n\frac{(\tilde{P}_n^{(\alpha,\beta)})^{(k)}(-1)}{n^{-\alpha+\beta+2k}}=\frac{\Gamma(\alpha+1)}{(-2)^k\Gamma(\beta+k+1)},
	\end{align}
and
\begin{align}\label{limiting value of norm of jacobi}
	\lim_{n\to+\infty}n^{2\alpha+1}\lVert\tilde{P}_n^{(\alpha,\beta)}\rVert^2_{\mu}=2^{\alpha+\beta}(\Gamma(\alpha+1))^2.
\end{align}
\begin{proof}
	Compute the $kth$ derivative of $\tilde{P}_n^{(\alpha,\beta)}(x)$ and  evaluate it at $x=-1$.
	\begin{align*}
		(\tilde{P}_n^{(\alpha,\beta)})^{(k)}(-1)&=\frac{\Gamma(\alpha+\beta+n+1+k)}{2^k\Gamma(\alpha+\beta+n+1)}\tilde{P}_{n-k}^{(\alpha+k,\beta+k)}(-1)\\
		&=\frac{n!}{(\alpha+1)_n}\frac{\Gamma(\alpha+\beta+n+1+k)}{2^k\Gamma(\alpha+\beta+n+1)}(-1)^{n-k}\binom{n+\beta}{n-k}\\
		&=\frac{\Gamma(n+1)\Gamma(\alpha+1)}{\Gamma(n+\alpha+1)}\frac{\Gamma(\alpha+\beta+n+1+k)}{2^k\Gamma(\alpha+\beta+n+1)}\frac{(-1)^{n-k}\Gamma(n+\beta+1)}{\Gamma(n-k+1)\Gamma(\beta+k+1)}\\
		&=\frac{1}{(-2)^k}\frac{\Gamma(\alpha+1)}{\Gamma(\beta+k+1)}\bigg(\frac{(-1)^n}{n^{\alpha-\beta-2k}}\frac{n^{\alpha}\Gamma(n+1)}{\Gamma(n+\alpha+1)}\frac{n^{-k}\Gamma(\alpha+\beta+n+1+k)}{\Gamma(\alpha+\beta+n+1)}\\
		&\times\frac{n^{-k-\beta}\Gamma(n+\beta+1)}{\Gamma(n-k+1)}\bigg).
	\end{align*}
Using \eqref{Gamma equality}, we obtain the required result \eqref{Asymptotes of Jacobi at -1}.
\begin{align*}
	\lVert\tilde{P}_n^{(\alpha,\beta)}\rVert^2_{\mu}&=\frac{2^{\alpha+\beta+1}}{2n+\alpha+\beta+1}\frac{\Gamma(n+\alpha+1)\Gamma(n+\beta+1)}{n!\Gamma(n+\alpha+\beta+1)}\bigg(\frac{n!}{(\alpha+1)_n}\bigg)^2\\
	&=\frac{2^{\alpha+\beta+1}\Gamma(n+1)\Gamma(n+\beta+1)\Gamma(n+\alpha+1)(\Gamma(\alpha+1))^2}{(2n+\alpha+\beta+1)\Gamma(n+\alpha+\beta+1)(\Gamma(n+\alpha+1))^2}\\
	&=\frac{2^{\alpha+\beta+1}(\Gamma(\alpha+1))^2}{n^{2\alpha+1}}\frac{n^{\alpha+\beta}\Gamma(n+1)}{\Gamma(n+\alpha+\beta+1)}\frac{n^{\alpha-\beta}\Gamma(n+\beta+1)}{\Gamma(n+\alpha+1)}\frac{n}{2n+\alpha+\beta+1}
\end{align*}
By using \eqref{Gamma equality}, we arrive at the desired outcome stated in \eqref{limiting value of norm of jacobi}.
\end{proof}
\end{lemma}
\begin{lemma} \label{derivative of CD kernel}
For $k, s \geq 0$, the derivatives of the Christoffel-Darboux kernel for Jacobi polynomials at $x = a = -1$ satisfy
	\begin{align*}
		\lim_{n\to+\infty}\frac{K_{n-1}^{(k,s)}(-1,-1)}{n^{2\beta+2k+2s+2}}=\frac{(-1)^{k+s}}{2^{\alpha+\beta+k+s+1}(\beta+k+s+1)\Gamma(\beta+k+1)\Gamma(\beta+s+1)}.
	\end{align*}
This scaling reflects the dominance of high-order terms in Sobolev norms as $n$ grows.
\begin{proof}
	Apply Stolz’s criterion to convert the limit into a difference quotient
	\begin{align*}
		\lim_{n\to+\infty}\frac{K_{n-1}^{(k,s)}(-1,-1)}{n^{2\beta+2k+2s+2}}&=\lim_{n\to+\infty}\frac{K_{n-1}^{(k,s)}(-1,-1)-K_{n-2}^{(k,s)}(-1,-1)}{n^{2\beta+2k+2s+2}-(n-1)^{2\beta+2k+2s+2}}.
	\end{align*}
Apply the Christoffel–Darboux formula to simplify the numerator, and use the binomial expansion to reduce the denominator.
	\vspace{-0.3cm}
\begin{align*}
		\lim_{n\to+\infty}\frac{K_{n-1}^{(k,s)}(-1,-1)}{n^{2\beta+2k+2s+2}}&=\lim_{n\to+\infty}\frac{\dfrac{(\tilde{P}_{n-1}^{(\alpha,\beta)})^{(k)}(-1)(\tilde{P}_{n-1}^{(\alpha,\beta)})^{(s)}(-1)}{\lVert \tilde{P}_{n-1}^{(\alpha,\beta)}\rVert^2_{\mu}}}{(2\beta+2k+2s+2)n^{2\beta+2k+2s+1}}\\
		&=\frac{1}{2\beta+2k+2s+2}\lim_{n\to+\infty}\frac{(\tilde{P}_{n-1}^{(\alpha,\beta)})^{(k)}(-1)}{n^{-\alpha+\beta+2k}}\frac{(\tilde{P}_{n-1}^{(\alpha,\beta)})^{(s)}(-1)}{n^{-\alpha+\beta+2s}}\frac{n^{-2\alpha-1}}{\lVert \tilde{P}_{n-1}^{(\alpha,\beta)}\rVert^2_{\mu}}.
\end{align*}
By employing \eqref{Asymptotes of Jacobi at -1} and \eqref{limiting value of norm of jacobi}, we arrive at the desired outcome.
\end{proof}
\end{lemma}
\begin{proposition}\label{Ratio of derivative of Sobolev and Jacobi}
	Let $j$ be a non-negative integer. Then, we have
	\begin{align*}
		\lim_{n\to+\infty}\frac{(Q_n^{M,N})^{(j)}(-1)}{(\tilde{P}_n^{(\alpha,\beta)})^{(j)}(-1)}=\frac{j(j-1)}{(\beta+j+1)(\beta+j+2)}
	\end{align*}
\begin{proof}
	Using value of $Q_n^{M,N}(x)$ from \eqref{Sobolev orthogonal}, we have
	{\small{\begin{align*}
				&\lim_{n\to+\infty}\frac{(Q_n^{M,N})^{(j)}(-1)}{(\tilde{P}_n^{(\alpha,\beta)})^{(j)}(-1)} \\
				&\!=\!\lim_{n\to+\infty} 1\! 
				-\! \frac{M \!\left( \tilde{P}_n^{(\alpha,\beta)}(-1)(1\! +\! N K_{n-1}^{(1,1)}(-1,\!-1))\! -\! N (\tilde{P}_n^{(\alpha,\beta)})'(-1) \!K_{n-1}^{(1,0)}(-1,\!-1) \right)\! K_{n-1}^{(0,j)}(-1,-1)}{\mathcal{D}\cdot (\tilde{P}_n^{(\alpha,\beta)})^{(j)}(-1)} \\
				& - \frac{N \left( (1 + M K_{n-1}(-1,-1))(\tilde{P}_n^{(\alpha,\beta)})'(-1) - M K_{n-1}^{(1,0)}(-1,-1) \tilde{P}_n^{(\alpha,\beta)}(-1) \right) K_{n-1}^{(1,j)}(-1,-1)}{\mathcal{D}\cdot (\tilde{P}_n^{(\alpha,\beta)})^{(j)}(-1)}, \\
				&\text{where } \mathcal{D} = 1 + N K_{n-1}^{(1,1)}(-1,-1) + M K_{n-1}(-1,-1) + MN K_{n-1}(-1,-1) K_{n-1}^{(1,1)}(-1,-1) \\
				&\quad - MN \left( K_{n-1}^{(1,0)}(-1,-1) \right)^2.
	\end{align*}}}
	We denote the expression
	\[
	M \left( \tilde{P}_n^{(\alpha,\beta)}(-1)\left(1 + N K_{n-1}^{(1,1)}(-1,-1)\right) - N \left(\tilde{P}_n^{(\alpha,\beta)}\right)'(-1) K_{n-1}^{(1,0)}(-1,-1) \right) K_{n-1}^{(0,j)}(-1,-1)
	\]
	as $\mathcal{T}_1$, and the expression
	\[
	N \left( \left(1 + M K_{n-1}(-1,-1)\right)\left(\tilde{P}_n^{(\alpha,\beta)}\right)'(-1) - M K_{n-1}^{(1,0)}(-1,-1) \tilde{P}_n^{(\alpha,\beta)}(-1) \right) K_{n-1}^{(1,j)}(-1,-1)
	\]
	as $\mathcal{T}_2$.
	Now, we analyze the asymptotic behavior term by term using Lemmas~\ref{Derivative and norm of Jacobi polynomials} and \ref{derivative of CD kernel}.
	\begin{align*}
		\tilde{P}_n^{(\alpha,\beta)}(-1) &\sim \frac{(-1)^n n^{-\alpha+\beta}\Gamma(\alpha+1)}{\Gamma(\beta + 1)}, \quad(\tilde{P}_n^{(\alpha,\beta)})'(-1) \sim \frac{(-1)^n n^{-\alpha+\beta + 2}\Gamma(\alpha+1)}{(-2)\Gamma(\beta + 2)}, \\
		(\tilde{P}_n^{(\alpha,\beta)})^{(j)}(-1) &\sim \frac{(-1)^n n^{-\alpha+\beta + 2j}\Gamma(\alpha+1)}{(-2)^j\Gamma(\beta + j + 1)}, \quad
		K_{n-1}(-1,-1) \sim \frac{n^{2\beta + 2}}{2^{\alpha+\beta + 1} (\beta + 1) \Gamma(\beta + 1)^2}, \\
		K_{n-1}^{(1,0)}(-1,-1) &\sim \frac{(-1)n^{2\beta + 4}}{2^{\alpha+\beta + 2} (\beta + 2) \Gamma(\beta + 1) \Gamma(\beta + 2)}, \\
		K_{n-1}^{(1,1)}(-1,-1) &\sim \frac{n^{2\beta + 6}}{2^{\alpha+\beta + 3} (\beta + 3) \Gamma(\beta + 2)^2}, \\
		K_{n-1}^{(0,j)}(-1,-1) &\sim \frac{(-1)^j n^{2\beta + 2j + 2}}{2^{\alpha+\beta + j + 1} (\beta + j + 1) \Gamma(\beta + 1) \Gamma(\beta + j + 1)}\quad\text{and} \\
		K_{n-1}^{(1,j)}(-1,-1) &\sim \frac{(-1)^{j+1} n^{2\beta + 2j + 4}}{2^{\alpha+\beta + j + 2} (\beta + j + 2) \Gamma(\beta + j + 1) \Gamma(\beta + 2)}.
	\end{align*}
	The dominant term in $\mathcal{D}$ is
	\begin{align*}
		\mathcal{D} \sim MN\left( K_{n-1}(-1,-1) K_{n-1}^{(1,1)}(-1,-1)-\left( K_{n-1}^{(1,0)}(-1,-1) \right)^2\right), 
	\end{align*}
	which simplifies to
	\begin{align*}
		\mathcal{D}\sim MN\frac{n^{4\beta+8}}{\left(2^{\alpha+\beta+2}\Gamma(\beta+1)\Gamma(\beta+2)
			\right)^2}\left(\frac{1}{(\beta+1)(\beta+3)}-\frac{1}{(\beta+2)^2}\right).
	\end{align*}
	The determinant factor evaluates to $1/(\beta+1)(\beta+3)(\beta+2)^2$, leading to
	\begin{align*}
		\mathcal{D}\cdot (\tilde{P}_n^{(\alpha,\beta)})^{(j)}(-1) \sim\frac{MN(-)^j(-1)^nn^{-\alpha+5\beta+2j+8}\Gamma(\alpha+1)}{2^j\left(2^{\alpha+\beta+2}\Gamma(\beta+1)\Gamma(\beta+2)
			\right)^2(\beta+1)(\beta+3)(\beta+2)^2\Gamma(\beta+j+1)}.
	\end{align*}
	Substituting the asymptotic expansions into $\mathcal{T}_1$ and $\mathcal{T}_2$, retaining leading-order terms, and simplifying, we find
	\begin{align*}
		\mathcal{T}_1&\sim-\frac{MN(-1)^j (-1)^nn^{-\alpha+5\beta+2j+8}\Gamma(\alpha+1)}{2^{2\alpha+2\beta+j+4}\left(\Gamma(\beta+1)\right)^2(\Gamma(\beta+2))^2(\beta+2)(\beta+3)(\beta + j + 1) \Gamma(\beta + j + 1)},\\
		\mathcal{T}_2&\sim\frac{MN(-1)^j (-1)^nn^{-\alpha+5\beta+2j+8}\Gamma(\alpha+1)}{2^{2\alpha+2\beta+j+4}\left(\Gamma(\beta+2)\right)^2\left(\Gamma(\beta+1)\right)^2(\beta+1)(\beta+2)(\beta + j + 2) \Gamma(\beta + j + 1)}.
	\end{align*}
	By combining these terms and dividing by the leading term of the denominator, we obtain the desired result.
\end{proof}
\end{proposition}
\begin{proposition}\rm(Norm asymptotics).\label{Norm asymptotics}
	The Sobolev norm satisfy
	\begin{align*}
	\lim_{n\to+\infty}\frac{\lVert Q_n^{(M,N)}\rVert_S}{\lVert \tilde{P}_n^{(\alpha,\beta)}\rVert_{\mu}}=1
	\end{align*}
\begin{proof}The squared Sobolev norm of $Q_n^{(M,N)}(x)$ is
	{\small{\begin{align*}
				\lVert Q_n^{(M,N)}\rVert^2_S
				&=\!\langle Q_n^{(M,N)}(x), \tilde{P}_n^{(\alpha,\beta)}(x)\rangle\!+\!MQ_n^{(M,N)}(-1)\tilde{P}_n^{(\alpha,\beta)}(-1)\!+\!N(Q_n^{(M,N)})'(-1)(\tilde{P}_n^{(\alpha,\beta)})'(-1).
	\end{align*}}}
	Dividing by $\lVert \tilde{P}_n^{(\alpha,\beta)}\rVert^2$
	\begin{align}\label{ratio of Sobolev and original Jacobi}
		&\frac{\lVert Q_n^{(M,N)}\rVert^2_S}{\lVert \tilde{P}_n^{(\alpha,\beta)}\rVert^2}=1+M\frac{Q_n^{(M,N)}(-1)\tilde{P}_n^{(\alpha,\beta)}(-1)}{\lVert \tilde{P}_n^{(\alpha,\beta)}\rVert^2}+N\frac{(Q_n^{(M,N)})'(-1)(\tilde{P}_n^{(\alpha,\beta)})'(-1)}{\lVert \tilde{P}_n^{(\alpha,\beta)}\rVert^2}.
	\end{align}
	From the construction of $Q_n^{M,N}(x)$, we use
	\begin{align*}
		&Q_n^{M,N}(-1)=\frac{\tilde{P}_n^{(\alpha,\beta)}(-1)(1+NK_{n-1}^{(1,1)}(-1,-1))-N(\tilde{P}_n^{(\alpha,\beta)})'(-1)K_{n-1}^{(1,0)}(-1,-1)}{\mathcal{D}},\\
		&(Q_n^{M,N})'(-1)=\frac{(1+MK_{n-1}(-1,-1))(\tilde{P}_n^{(\alpha,\beta)})'(-1)-MK_{n-1}^{(1,0)}(-1,-1)\tilde{P}_n^{(\alpha,\beta)}(-1)}{\mathcal{D}},\\
		&\text{where } \mathcal{D} = 1 + N K_{n-1}^{(1,1)}(-1,-1) + M K_{n-1}(-1,-1) + MN \left(K_{n-1} K_{n-1}^{(1,1)}-\left( K_{n-1}^{(1,0)} \right)^2\right).
	\end{align*}
	Using Lemmas~\ref{Derivative and norm of Jacobi polynomials} and ~\ref{derivative of CD kernel}, substitute the leading asymptotic terms
	\begin{align*}
		\tilde{P}_n^{(\alpha,\beta)}(-1)\sim\mathcal{O}(n^{-\alpha+\beta}),& \quad (\tilde{P}_n^{(\alpha,\beta)})'(-1)\sim\mathcal{O}(n^{-\alpha+\beta+2})\\
		K_{n-1}(-1,-1)\sim \mathcal{O}(n^{2\beta+2}),&\quad K_{n-1}^{(1,0)}(-1,-1)\sim \mathcal{O}(n^{2\beta+4}),\\
		K_{n-1}^{(1,1)}(-1,-1)\sim \mathcal{O}(n^{2\beta+6}),&\quad \lVert \tilde{P}_n^{(\alpha,\beta)}\rVert^2\sim\mathcal{O}(n^{-2\alpha-1}),
	\end{align*}
	where $\mathcal{O}$ is Big O notation. The leading term in $\mathcal{D}$ is $\mathcal{D}\sim\mathcal{O}(n^{4\beta+8})$. Substituting these asymptotics into \eqref{ratio of Sobolev and original Jacobi}
	\begin{align*}
		\frac{\lVert Q_n^{(M,N)}\rVert^2_S}{\lVert P_n^{(\alpha,\beta)}\rVert^2}=1+\frac{1}{\mathcal{O}(n)}+\frac{1}{\mathcal{O}(n)}.
	\end{align*}
	We obtain the required result as $n\to\infty$.
\end{proof}
\end{proposition}
The figure~\ref{Fig1_r_d_s_j} depicts the convergence of the ratio $\frac{(Q_n^{M,N})^{(j)}(-1)}{(\tilde{P}_n^{(\alpha,\beta)})^{(j)}(-1)}$ as $n$ increases, for the case $\alpha=0$, $\beta=1$ and $j=2$, verifying Proposition~\ref{Ratio of derivative of Sobolev and Jacobi}. As $n$ increases, the ratio approaches the red dashed line, confirming the theoretical prediction. The figure~\ref{Fig2_r_n_s_j} illustrates the convergence of the ratio $\frac{\lVert Q_n^{(M,N)}\rVert_S}{\lVert \tilde{P}_n^{(\alpha,\beta)}\rVert_{\mu}}$ as $n$ increases, confirming the asymptotic result stated in Proposition~\ref{Norm asymptotics}, which says that the Sobolev norm of $Q_n^{(M,N)}$ is asymptotically equal to the norm of $\tilde{P}_n^{(\alpha,\beta)}$ as $n\to\infty$.


\begin{figure}[htbp]
	\centering
	
	\begin{subfigure}{0.45\textwidth}
		\centering
		\includegraphics[width=\linewidth]{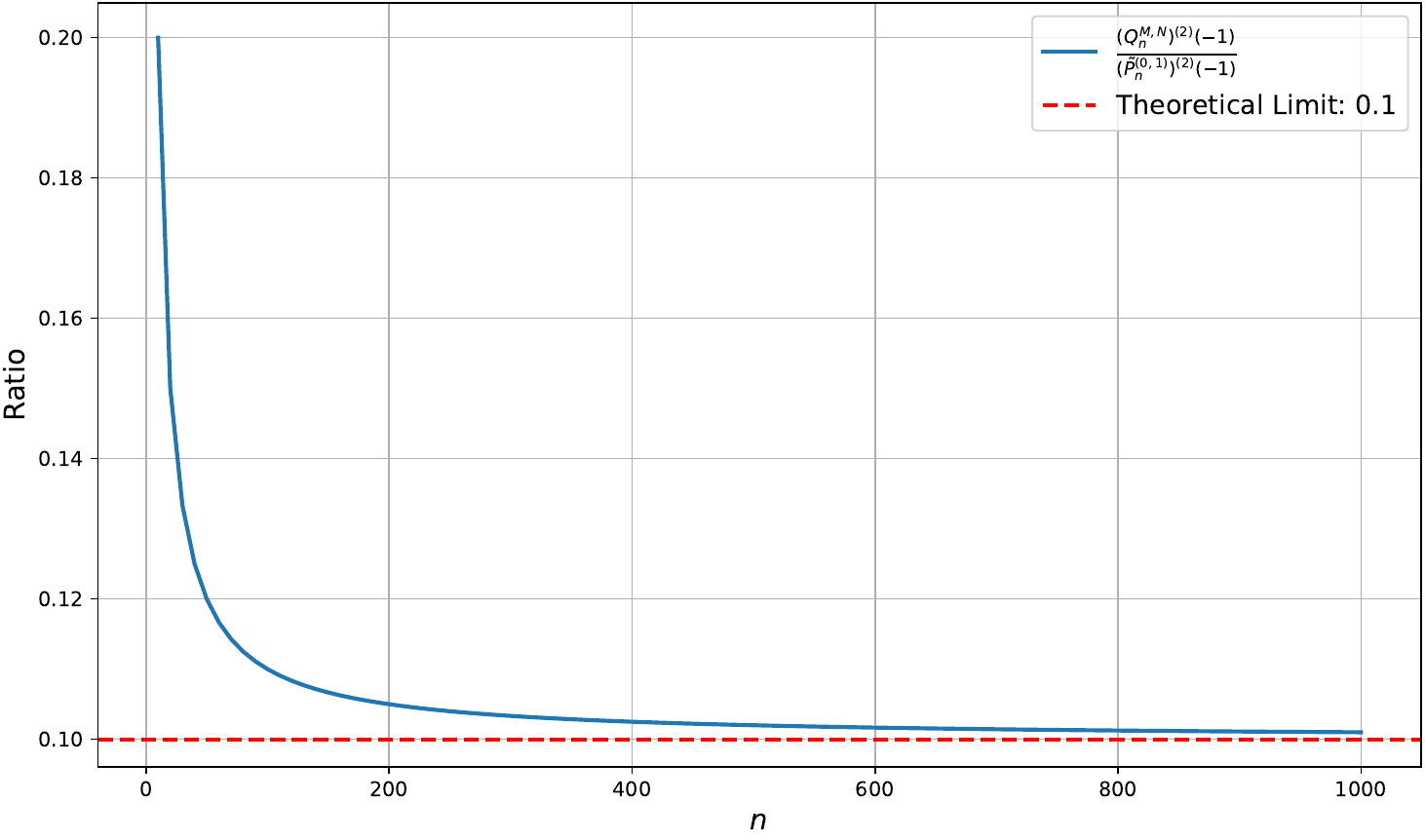}
		\caption{Convergence of the Ratio to the Theoretical Limit}
		\label{Fig1_r_d_s_j}
	\end{subfigure}
	\hfil
	\begin{subfigure}{0.45\textwidth}
		\centering
		\includegraphics[width=\linewidth]{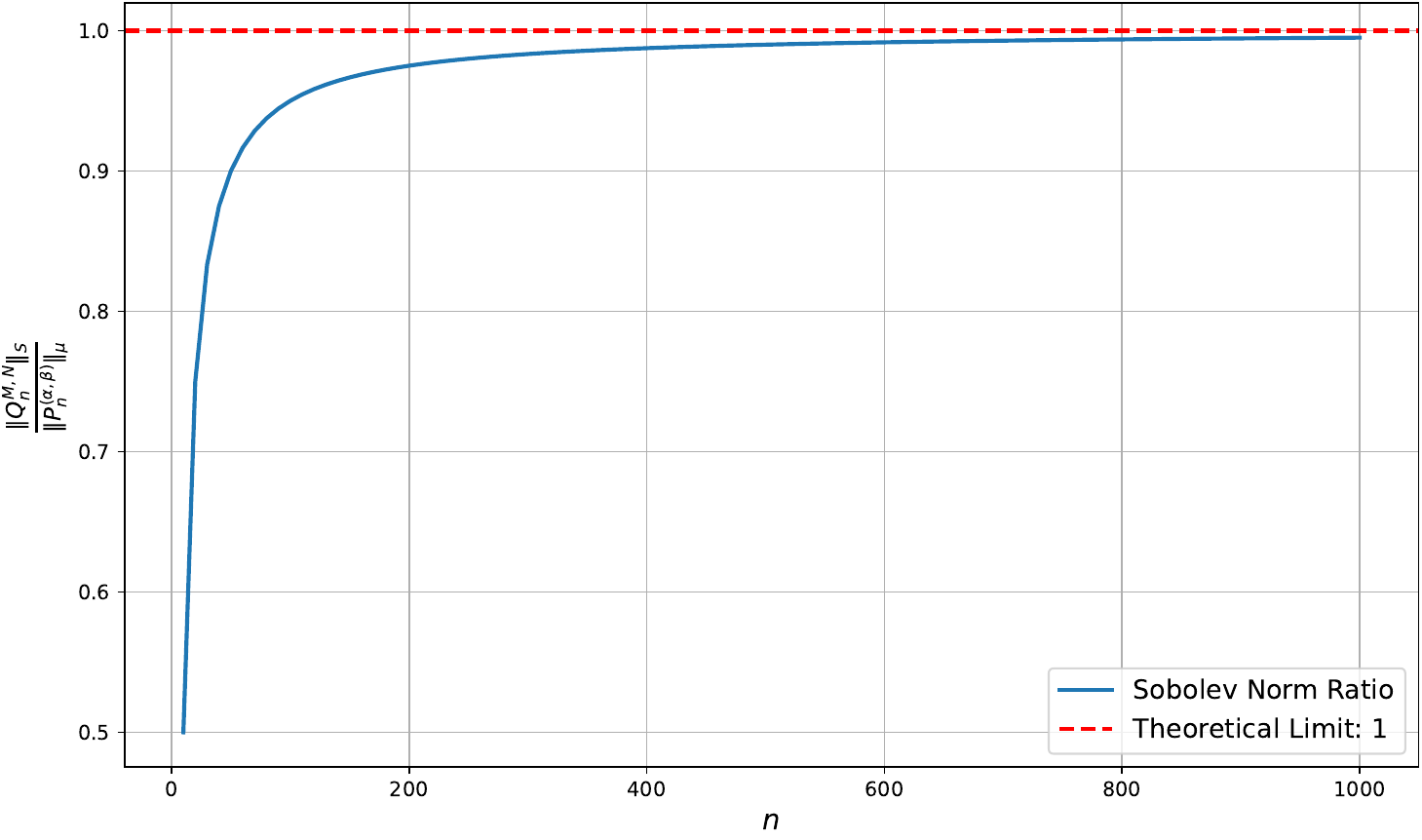}
		\caption{Convergence of Sobolev Norms to Jacobi Norms}
		\label{Fig2_r_n_s_j}
	\end{subfigure}

\end{figure}

\textbf{Acknowledgment:} This research did not receive any specific grant from funding agencies in the public, commercial, or not-for-profit sectors.



\end{document}